\documentclass{amsart}
\usepackage{amsmath}
\usepackage{amssymb}
\usepackage{mathrsfs}
\usepackage{pgf,tikz}
\usepackage{mhchem}
\usetikzlibrary{positioning}
\usetikzlibrary{shadings}
\usetikzlibrary{arrows,automata}
\usepackage{graphicx}
\usepackage{systeme,mathtools}
\usepackage{amsthm,enumitem}
\usepackage{systeme}
\usepackage{bigfoot}
\input xy
\xyoption{all}

\newtheorem{theorem}{Theorem}[section]

\newtheorem{lemma}[theorem]{Lemma}
\newtheorem{corollary}[theorem]{Corollary}
\newtheorem{proposition}[theorem]{Proposition}
\theoremstyle{definition}
\newtheorem{definition}[theorem]{Definition}
\newtheorem{remark}{Remark}

\numberwithin{equation}{section}

\begin{document}
	\title[Lie and Jordan structures of Leavitt path algebras]{A note on Lie and Jordan structures of\\ Leavitt path algebras}
	\author[Huynh Viet Khanh]{Huynh Viet Khanh$  ^\dagger$}\thanks{$^\dagger$ Corresponding author}
	\address{Department of Mathematics and Informatics, HCMC University of Education, 280 An Duong Vuong Str., Cho Quan Ward, Ho Chi Minh City, Vietnam}
	\email{khanhhv@hcmue.edu.vn}
	\author[Le Qui Danh]{Le Qui Danh}
	\address{University of Architecture Ho Chi Minh City, Ho Chi Minh City, Vietnam}
	\email{danh.lequi@uah.edu.vn}
	
	
	\maketitle
	\begin{abstract} 
		Let $L_K(E)$ be the Leavitt path algebra of a directed graph $E$ over a field $K$. In this paper, we determine $E$ and $K$ for the Lie algebra $\mathbf{K}_{L_K(E)}$ and the Jordan algebra $\mathbf{S}_{L_K(E)}$ arising from $L_K(E)$ with respect to the standard involution to be solvable. 
	\end{abstract}
	
	\section{Introduction}
	Every associative algebra $\mathcal{A}$ over a field $K$ can give rise to two new structures: the \textit{Lie algebra}, denoted by $\mathcal{A}^-$, and the \textit{Jordan algebra}, denoted by $\mathcal{A}^\circ$ over $K$. These structures can be obtained by defining new products within $\mathcal{A}$, namely the Lie bracket $[x,y]=xy-yx$ for Lie algebras and the Jordan product $x \circ y = xy+yx$ for Jordan algebras respectively, for every $x,y\in  \mathcal{A}$. A reasonable question arises: Are the properties of $\mathcal{A}$ as a ring mirrored in its Lie or Jordan algebra counterparts? Specifically, does the simplicity of $\mathcal{A}$ as a ring guarantee its simplicity as a Lie or Jordan algebra? In a seminal paper \cite{Pa_herstein_61} by Herstein, this question was answered affirmatively. Assume further that $\mathcal{A}$ is equipped with an involution $\star$; that is, a map $^\star: \mathcal{A}\to  \mathcal{A}$ satisfying the following conditions:
	\begin{enumerate}
		\item $(a+b)^{\star}=a^{\star}+b^{\star}$,
		\item $(ab)^{\star}=b^{\star}a^{\star}$,
		\item $(a^{\star})^{\star}=a$,
	\end{enumerate}
	for all $a,b\in \mathcal{A}$. Let 
	$$
	\mathbf{S}_{\mathcal{A}}=\{x\in \mathcal{A} : x^{\star}=x\} \text{ and } \mathbf{K}_{\mathcal{A}}=\{x\in \mathcal{A} : x^{\star}=-x\}
	$$
	be the sets of \textit{symmetric elements} and \textit{skew-symmetric elements} of $\mathcal{A}$ with respect to ${\star}$, respectively. It is obvious that $\mathbf{K}_{\mathcal{A}}$ and $\mathbf{S}_{\mathcal{A}}$ are Lie subalgebra and Jordan subalgebra of $\mathcal{A}^-$ and $\mathcal{A}^\circ$ respectively. These are the most important structures of Lie or Jordan algebras arising from an associative algebra. In particular, when $\mathcal{A}$ is the full matrix ring $\mathbb{M}_n(K)$ and $^\star$ is the transpose operation, then $\mathbf{K}_{\mathcal{A}}$ is identified with the Lie algebra of \textit{skew-symmetric matrices} which corresponds to the \textit{Lie group of orthogonal matrices}. Under certain assumptions on $\mathcal{A}$, Lie and Jordan structures of $\mathbf{K}_{\mathcal{A}}$ and  $\mathbf{S}_{\mathcal{A}}$ respectively were intensively investigated by Herstein in a series of papers (\cite{Pa_herstein_55_a}-\cite{Pa_herstein_61}) and by other authors in the references therein.
	
	In this paper, we consider the case when $\mathcal{A}$ is the Leavitt path algebra $L_K(E)$ of a directed graph $E$ over a field $K$. The concept of the Leavitt path algebra $L_K(E)$ associated with a graph $E$ and coefficients from a field $K$ made its debut in \cite{Pa_abrams-pino-05}, sparking widespread interest among scholars. These algebras were concurrently and separately introduced in \cite{Pa_ara-moreno-pardo-07}, despite being published two years following \cite{Pa_abrams-pino-05}. Within the past few years, the Lie subalgebras of $L_K(E)^-$ arising from $L_K(E)$ have been investigated by many authors. Motivated by the groundbreaking work of Herstein \cite{Pa_herstein_61}, necessary and sufficient conditions for the Lie algebra $[L_K(E),L_K(E)]$ of $L_K(E)^-$ to be simple were given in \cite{Pa_abrams-mesyan-12} and \cite{Pa_alahmedi-alsulami-16}, under the assumption that $E$ is row-finite. In another direction, the paper \cite{nam-zhang-22} completely identified the graph $E$ and the field $K$ for the Lie algebra $L_K(E)^-$ to be solvable. 
	
	Let $E$ be a graph and $K$ a field. Then, the mapping $^\star: L_K(E)\to L_K(E)$ which sends $v$ to $v$, $e$ to $e^*$ and $e^*$ to $e$ for all  $v\in E^0$ and $e\in E^1$, and leaves all elements in $K$ fixed, extends to an involution of $L_K(E)$. We call this the \textit{standard involution} on $L_K(E)$ (see also \cite[page 36]{Bo_abrams_2017}). Let $\mathbf{K}_{L_K(E)}$ and $\mathbf{S}_{L_K(E)}$ be the Lie subalgebra of $L_K(E)^-$ and the Jordan subalgebra of $L_K(E)^\circ$ with respect to this involution, respectively. In \cite{Pa_alahmadi-alsulami-16-2} the simplicity of the Lie algebra $\mathbf{K}_{L_K(E)}$ was also studied. The current paper is devoted to a study of the Lie algebra $\mathbf{K}_{L_K(E)}$ and the Jordan algebra $\mathbf{S}_{L_K(E)}$. To be more specific, we identify the graph $E$ and the field $K$ for the Lie algebra $\mathbf{K}_{L_K(E)}$ and the Jordan algebra $\mathbf{S}_{L_K(E)}$ to be solvable. 
	
	\section{Preliminaries}
	\subsection{Directed graphs and Leavitt path algebras.}\label{subsection_LPA}
	Throughout this paper, the basic notations and conventions are taken from \cite{Bo_abrams_2017}. In particular, a (directed) graph $E=(E^0, E^1, r, s)$ consists of two sets $E^0$ and $E^1$ together with maps $r, s: E^1\to E^0$. The elements of $E^0$ are called \textit{vertices} and the elements of $E^1$ \textit{edges}. For the convenience of readers, in the following, we outline some concepts and notations that we use substantially in this paper. 
	
	Let $E$ be a graph. A vertex $v$ in $E$ is a \textit{sink} if it emits no edges, while it is an \textit{infinite emitter} if it emits infinitely many edges. A graph $E$ is said to be \textit{row-finite} if $E$ contains no infinite emitter. A vertex $v$ is said to be \textit{regular} if it is neither a sink nor an infinite emitter. A \textit{path} $\mu$ in a graph $E$ is a finite sequence of edges $\mu=e_1e_2\dots e_n$ with $r(e_i) =s(e_{i+1})$ for all $1\leq i\leq n-1$. We set $s(\mu):=s(e_1)$ and $r(\mu):=r(e_n)$. The set of all finite paths in $E$ is denoted by ${\rm Path}(E)$. An \textit{exit} for a path $\mu=e_1\dots e_n$ is an edge $e$ such that $s(e)=s(e_i)$ for some $i$ and $e\ne e_i$.

	\begin{definition}[Leavitt path algebra]\label{definition_1.3}
		Let $E$ be a graph and $K$ a field. We define a set $\left( E^1\right)^*$ consisting of symbols of the form $\{e^*| e\in E^1\}$.   The \textit{Leavitt path algebra of $E$ with coefficients in $K$}, denoted by $L_K(E)$, is the free associative $K$-algebra generated by $E^0\cup E^1\cup (E^1)^*$, subject to the following relations:
		\begin{enumerate}[]
			\item[(V)] $vv'=\delta_{v,v'}v$ for all $v,v'\in E^0$,
			\item[(E1)] $s(e)e=er(e)=e$ for all $e\in E^1$,
			\item[(E2)] $r(e)e^*=e^*s(e)=e^*$ for all $e\in E^1$,
			\item[(CK1)] $e^*f=\delta_{e,f}r(e)$ for all $e,f\in E^1$, and
			\item[(CK2)] $v=\sum_{\{e\in E^1|s(e)=v\}}ee^*$ for every $v\in{\rm Reg}(E)$.
		\end{enumerate}
		
		Let $E$ be a graph. For each $e\in E^1$, we call $e^*$ the \textit{ghost edge}. Also, the notation $\mu^*=e_n^*\dots e_2^*e_1^*$ stands for the corresponding ghost path of the path $\mu=e_1\dots e_n$.
	\end{definition}
	
	\subsection{Lie and Jordan structures of associative algebras}
	Every associative algebra $\mathcal{A}$  over a field $K$ is endowed with a Lie algebra structure for the \textit{Lie bracket} $[a,b]=ab-ba$ for all $a,b\in \mathcal{A}$. We denote this Lie algebra by  $ \mathcal{A}^-$. For $K$-subspaces $S$ and $T$ of $\mathcal{A}$, we define $[S,T]={\rm span}_K\{ ab-ba : a\in S, b\in T \}$. A $K$-subspace $\mathcal{L}$ of $\mathcal{A}$ is called a \textit{Lie subalgebra} of $\mathcal{A}^-$ if $[a,b]\in \mathcal{L}$ for all $a,b\in \mathcal{L}$.  For the Lie subalgebra $\mathcal{L}$ of $\mathcal{A}^-$, we define inductively the two series
	$$
	\mathcal{L}^{(0)}=\mathcal{L}, \; \mathcal{L}^{(1)}=[\mathcal{L},\mathcal{L}], \; \dots, \; \mathcal{L}^{(n)}=[\mathcal{L}^{(n-1)},\mathcal{L}^{(n-1)}],\dots
	$$
	and
	$$
	\mathcal{L}^0=\mathcal{L}, \; \mathcal{L}^0=[\mathcal{L},\mathcal{L}], \; \dots, \; \mathcal{L}^n=[\mathcal{L}^{n-1},L],\dots
	$$
	The Lie subalgebra $\mathcal{L}$ is called \textit{Lie solvable} (resp., \textit{Lie nilpotent}) if there exists $n\geq0$ (resp., $m\geq0$) such that $\mathcal{L}^{(n)}=0$ (resp., $\mathcal{L}^m=0$). If such a smallest $n$ (resp., $m$) exists, we will say that $ \mathcal{L}$ is \textit{Lie solvable of index} $n$ (resp., \textit{Lie nilpotent of index} $m$).
	
	Similarly, a \textit{Jordan product} can be defined on $\mathcal{A}$ by $a \circ b = ab+ba$, for all $a,b\in \mathcal{A}$. If $\mathcal{A}$ is viewed as a Jordan algebra for the product $ \circ $, we denote it by $\mathcal{A}^\circ$. Note that in the case ${\rm char}(K)=2$, the two structures $\mathcal{A}^\circ$ and $\mathcal{A}^-$ coincide. A $K$-subspace $\mathcal{J}$ of $\mathcal{A}$ is called a \textit{Jordan subalgebra} of $\mathcal{A}^\circ$ if $a\circ b \in \mathcal{J}$ for all $a,b\in \mathcal{J}$. 
	For two $K$-subspaces $S$ and $T$ of $\mathcal{A}$, we define
	$$
	S\circ T={\rm span}_K\{ ab+ba : a\in S, b\in T \}.
	$$ 
	For the Jordan subalgebra $\mathcal{J}$ of $\mathcal{A}^\circ$, we define the analogous series for $\mathcal{J}$ as follows: 
	$$
	\mathcal{J}_{(0)}=\mathcal{J}, \; \mathcal{J}_{(1)}=\mathcal{J}\circ \mathcal{J}, \; \dots, \; \mathcal{J}_{(n)}=\mathcal{J}_{(n-1)}\circ \mathcal{J}_{(n-1)},\dots
	$$
	and 
	$$
	\mathcal{J}_0=\mathcal{J}, \; \mathcal{J}_1=\mathcal{J}\circ \mathcal{J}, \; \dots, \; \mathcal{J}_n=\mathcal{J}_{n-1}\circ \mathcal{J},\dots
	$$
	The Jordan subalgebra $\mathcal{J}$ is called \textit{Jordan solvable} (resp., \textit{Jordan nilpotent}) if there exists $n\geq0$ (resp., $m\geq0$) such that $\mathcal{J}_{(n)}=0$ (resp., $\mathcal{J}_m=0$). The \textit{index of a solvable Jordan algebra} is defined similarly.
	
	In this paper, we particularly consider the special case where $\mathcal{A} = L_K(E)$, $\mathcal{L}=\mathbf{K}_{L_K(E)}$ and $\mathcal{J}=\mathbf{S}_{L_K(E)}$ with respect to the standard involution on $L_K(E)$.
	
	\subsection{Homomorphisms of algebras with involutions}
	Let $K$ be a field, and $(\mathcal{A},^\star)$ and $(\mathcal{B},^ \natural)$ be $K$-algebras with involutions. \textit{A homomorphism of $K$-algebras with involution} $\varphi: (\mathcal{A},^\star)\to (\mathcal{B},^\natural)$ is a $K$-algebra homomorphism $\varphi: \mathcal{A}\to \mathcal{B}$ such that $(\varphi(ab))^\natural=\varphi(b^\star)\varphi(a^\star)$. 
	
	It is obvious that if $(\mathcal{A},^\star)$ is isomorphic to $(\mathcal{B},^\natural)$ as $K$-algebras with involution, then the Lie subalgebras $\mathbf{K}_{\mathcal{A}}$ and $\mathbf{K}_{\mathcal{B}}$, with respect to the involutions $^\star$ and $\natural$ respectively, are also isomorphic as Lie algebras over $K$.  Similarly, the Jordan subalgebras $\mathbf{S}_{\mathcal{A}}$ and $\mathbf{S}_{\mathcal{B}}$ of $\mathcal{A}^\circ$ and $\mathcal{B}^\circ$, with respect to the involutions $^\star$ and $\natural$ respectively, are also isomorphic as Jordan algebras over $K$.
	\section{Lie Solvability and Jordan Solvability of $\mathbb{M}_n(\mathcal{A})$}
	Let $K$ be a field, and $\mathcal{A}$  a unital $K$-algebra with involution $^\natural$. Let  $\mathbb{M}_n(\mathcal{A})$ be the matrix ring of degree $n\geq 1$ over $\mathcal{A}$. Then, it is easy to check that $^\natural$ is extended to an involution, which is again denoted by $^\natural$, on $\mathbb{M}_n(\mathcal{A})$ by letting 
	\begin{equation}\label{equation_involution}
		(a_{ij})_{n\times n}^\natural=(a_{ji}^\natural)_{n\times n}
	\end{equation}
	for every matrix $(a_{ij})\in \mathbb{M}_n(\mathcal{A})$. Let $\mathbf{K}_{\mathbb{M}_n(\mathcal{A})}$ 
	be the Lie subalgebra of $\mathbb{M}_n(\mathcal{A})^-$ with respect to the involution $\natural$ on $\mathbb{M}_n(\mathcal{A})$ as defined in (\ref{equation_involution}). In this section, we give conditions on $n$ and $\mathcal{A}$ for $\mathbf{K}_{\mathbb{M}_n(\mathcal{A})}$ to be Lie solvable. First, we fix the following notation: In  $\mathbb{M}_n(\mathcal{A})$, we denote by $E_{ij}$ the matrix whose {$(i,j)$}-entry is $1$ and the other entries are $0$. The following proposition is the main result of this section:
	
	\begin{proposition}\label{proposition _involution}
		Let $K$ be a field, and $\mathcal{A}$ be a unital $K$-algebra with an involution $^\natural$. Assume further that  $\mathcal{A}$  is an integral domain. Let $\mathbf{K}_{\mathbb{M}_n(\mathcal{A})}$ be as defined above. Then, the following assertions hold:
		\begin{enumerate}[font=\normalfont]
			\item[(a)] $\mathbf{K}_{\mathbb{M}_n(\mathcal{A})}$ is not Lie solvable when $n\geq3$.
			\item[(b)] If ${\rm char}(K)=2$, then $\mathbf{K}_{\mathbb{M}_2(\mathcal{A})}$ is not Lie nilpotent.
			\item[(c)] If ${\rm char}(K)=2$, then $\mathbf{K}_{\mathbb{M}_2(\mathcal{A})}$ is Lie solvable, with index at most $3$. If moreover, the involution on $\mathcal{A}$ is trivial, then the index is $2$.
			\item[(d)] Assume that ${\rm char}(K)\ne2$. If the involution on $\mathcal{A}$ is trivial, then $[\mathbf{K}_{\mathbb{M}_2(\mathcal{A})},\mathbf{K}_{\mathbb{M}_2(\mathcal{A})}]=0$; while if the involution on $\mathcal{A}$ is non-trivial, then $\mathbf{K}_{\mathbb{M}_2(\mathcal{A})}$ is not Lie solvable.
		\end{enumerate}
	\end{proposition}
	\begin{proof}
		(a): Assume first that $n\geq 3$. Let $a,b,c\in K$ be such that $c\neq 0$ and $a^2+b^2\ne 0$. We construct inductively a family of matrices $\{A_m, B_m, X_m\;|\; m=1,2,\dots\}$  in $\mathbb{M}_n(\mathcal{A})$ as follows: Put
		$$
		A_1=a(E_{12}-E_{21})+b(E_{13}-E_{31}),~~
		B_1=c(E_{23}-E_{32}),~~
		X_1=[A_1,B_1];
		$$
		and for each $m\geq 2$, we put 
		$$
		A_m=[X_{m-1},B_{m-1}],~~B_m=[X_{m-1},A_{m-1}], \text{ and }X_{m}=[A_m,B_m].
		$$
		Since $A_1,B_1\in\mathbf{K}_{\mathbb{M}_n(\mathcal{A})}$, we get that $X_1\in\mathbf{K}_{\mathbb{M}_n(\mathcal{A})}^{(1)}$. Inductively, it can be checked that $A_m, B_m\in \mathbf{K}_{\mathbb{M}_n(\mathcal{A})}^{(m-1)}$, from which it follows that $X_m\in  \mathbf{K}_{\mathbb{M}_n(\mathcal{A})}^{(m)}$. For each $m\geq1$, we claim that there exist $a_m,b_m,c_m\in K$ such that $c_m\neq 0$, $a_m^2+b_m^2 \ne 0$, and
		$$ 
		\begin{aligned}
			& A_m=a_m(E_{12}-E_{21})+b_m(E_{13}-E_{31}),\\
			& B_m=c_m(E_{23}-E_{32}),\\
			& X_m\neq 0.
		\end{aligned}
		$$
		The argument proceeds by induction on $m$. For $m=1$, then we set $a_1=a$, $b_1=b$ and $c_1=c$. Then, we have
		$$ 
		\begin{aligned}
			& A_1=a_1(E_{12}-E_{21})+b_1(E_{13}-E_{31}),\\
			& B_1=c_1(E_{23}-E_{32}),\\
			& X_1=-b_1c_1(E_{12}-E_{21})+a_1c_1(E_{13}-E_{31})\neq 0.
		\end{aligned}
		$$
		Thus, the claim is true for $m=1$. Now, assume that the claim is true for all $m\geq i$ for some $i\geq 2$. We will prove the claim is also true for $m=i+1$. Now, the inductive assumption shows that there exist $a_i,b_i,c_i\in K$ such that $c_i\neq 0, a_i^2+b_i^2\ne0$, and
		$$ 
		\begin{aligned}
			& A_i=a_i(E_{12}-E_{21})+b_i(E_{13}-E_{31}),\\
			& B_i=c_i(E_{23}-E_{32}),\\
			&X_i=[A_i,B_i]=-b_ic_i(E_{12}-E_{21})+a_ic_i(E_{13}-E_{31})\neq 0.
		\end{aligned}
		$$
		From this, a direct calculation shows that 
		$$ 
		\begin{aligned}
			&A_{i+1}=[X_i,B_i]=-a_ic_i^2(E_{12}-E_{21})-b_ic_i^2(E_{13}-E_{31}),\\
			&B_{i+1}=[X_i,A_i]=(a_i^2+b_i^2)c_i(E_{23}-E_{32}),\\
			&X_{i+1}=[A_{i+1},B_{i+1}]=b_i(a_i^2+b_i^2)c_i^3(E_{12}-E_{21})-a_i(a_i^2+b_i^2)c_i^3(E_{13}-E_{31}).
		\end{aligned}
		$$
		Since $a_i^2+b_i^2\ne0$,  we have $X_{i+1}\neq 0$. The claim is proved. 
		It follows that $X_m\neq 0$ for all $m\geq 1$, and so $\mathbf{K}_{\mathbb{M}_n(\mathcal{A})}^{(m)}\neq 0$ for all $m\geq 1$. Therefore $\mathbf{K}_{\mathbb{M}_n(K)}$ is not Lie solvable, proving (a).
		
		To prove the remaining assertions, we assume from now on that $n=2$. Let us calculate $\mathbf{K}_{\mathbb{M}_2(\mathcal{A})}$ more explicitly. It is obvious that 
		$$
			\mathbf{K}_{\mathbb{M}_2(\mathcal{A})}=\left\lbrace \begin{pmatrix}
				a & b \\
				-b^\natural & c
			\end{pmatrix} \;|\; a,b,c\in \mathcal{A} \text{ and } a=-a^\natural, c=-c^\natural \right\rbrace .
		$$
		
		Let $A_i, B_i$ ($1\leq i\leq 4$) be arbitrary matrices in $\mathbf{K}_{\mathbb{M}_2(\mathcal{A})}$. Then, these matrices can be written as
		\[
		A_i=\left(\begin{array}{cc}
			a_i & b_i\\
			-b_i^\natural & c_i
		\end{array}\right)
		\text{ and }
		B_i=\left(\begin{array}{cc}
			u_i & v_i\\
			-v_i^\natural & w_i
		\end{array}\right),
		\]
		where $a_i,b_i,c_i,u_i,v_i,w_i\in \mathcal{A}$ and $a_i=-a_i^\natural$, $c_i=-c_i^\natural$, $u_i=-u_i^\natural$, $w_i=-w_i^\natural$.
		
		\smallskip 
		
		For each $1\leq i\leq 4$, put $X_i=[A_i,B_i]$. Because $a_i=-a_i^\natural$, $c_i=-c_i^\natural$, $u_i=-u_i^\natural$, and $w_i=-w_i^\natural$, 
		we get
		\begin{equation}\label{equation_1}
			X_i=\left(\begin{array}{cc}
				r_i& s_i\\
				-s_i^\natural & -r_i
			\end{array}\right), \text{ where }
		\end{equation}
		\begin{equation}\label{equation_2}
			r_i=b_i^\natural v_i-b_iv_i^\natural \text{ and } s_i=v_i(a_i-c_i)+b_i(w_i-u_i).
		\end{equation}
		
		It follows that 
		\begin{equation}\label{equation_3}
			[X_1,X_2] =\left(\begin{array}{cc}
				s_2s_1^\natural -s_1s_2^\natural  & 0 \\[0.3cm]
				2(r_1s_2^\natural-s_1^\natural r)&  s_2^\natural s_1-s_1^\natural s_2
			\end{array}\right), \text{ where }
		\end{equation}
	   
	    \begin{equation}\label{equation_4}
	    	\begin{aligned}
	    		&	s_2s_1^\natural-s_1s_2^\natural\hspace*{0.1cm}			= 	(v_2v_1^\natural-v_1v_2^\natural)(a_1-c_1)(c_2-a_2)\\
	    		&\hspace*{1.9cm}+(v_2b_1^\natural-b_1v_2^\natural)(a_2-c_2)(u_1-w_1)\\
	    		&\hspace*{1.9cm}+(v_1b_2^\natural-b_2v_1^\natural)(w_2-u_2)(a_1-c_1)\\
	    		&\hspace*{1.9cm}+(b_1b_2^\natural-b_2b_1^\natural)(u_1-w_1)(u_2-w_2).
	    	\end{aligned}
	    \end{equation}

		(b)  and (c): Assume that ${\rm char}(K)=2$. As $2(r_1s_2^\natural-s_1^\natural r)=0$, it follows from (\ref{equation_3}) that $[X_1,X_2]$ is a diagonal matrix. Similarly, we also obtain that $[X_3,X_4]$ is also a diagonal matrix. Thus $[[X_1,X_2],[X_3,X_4]]=0$, which means that $\mathbf{K}_{\mathbb{M}_2(\mathcal{A})}^{(3)} = 0$. We divide our situation into two possible cases:
		
		\medskip 
		
		\textbf{Case 1.} \textit{${\rm char}(K)=2$ and the involution $^\natural$ on $\mathcal{A}$ is trivial.} As $b_i^\natural=b_i$ and $v_i^\natural=v_i$, we get from (\ref{equation_2}) that $r_i=0$. It follows from (\ref{equation_1}) that $X_i=\left(\begin{array}{cc}
			0& s_i\\
			-s_i^\natural & 0
		\end{array}\right)$, and so $[X_1,X_2]=0$; or equivalently, we have $\mathbf{K}_{\mathbb{M}_2(\mathcal{A})}^{(2)} = 0$. In other words, $\mathbf{K}_{\mathbb{M}_2(\mathcal{A})}$ is Lie solvable of index 2. This completes the proof of the second assertion of (c).
		
		\medskip 
		
		\textbf{Case 2.} \textit{${\rm char}(K)=2$ and the involution $^\natural$ on $\mathcal{A}$ is non-trivial.} In this case, suitable choices of the matrices $A_1, A_2, B_1, B_2$ yield that $[X_1,X_2]\ne0$. For example, if we choose $a_1=c_1$, $a_2=c_2$, $b_1=1$, $u_1\ne w_1$, $w_1=w_2=0$, $0\ne u_1=u_2\in K$, $b_2\ne b_2^\natural$, then it follows from (\ref{equation_4}) that $s_2s_1^\natural-s_1s_2^\natural =(b_2^\natural-b_2)u_1^2\ne 0$. Therefore, we get from (\ref{equation_3})  that $[X_1,X_2]\ne0$. It follows that $\mathbf{K}_{\mathbb{M}_2(\mathcal{A})}$ is Lie solvable of index 3. This completes the proof of (c).
		
		Now, we give the proof of (b). Let $A=E_{12}+E_{21}$ and $B=E_{11}$, which belong to $\mathbf{K}_{\mathbb{M}_2(\mathcal{A})}$. Define 
		$$
		[A,\ce{_m}B]= [...[A,{\! \underbrace{B],\cdots,B]\,}_\text{$m$ times}},
		$$
		for some $m\geq1$. Then, it is clear that $[A,\ce{_m}B]\in\mathbf{K}_{\mathbb{M}_2(\mathcal{A})}^m$ for all $m\geq1$. Moreover, it is straightforward to show that $[A,\ce{_m}B]=E_{12}+E_{21}\ne 0$ for all $m\geq1$. This implies that $\mathbf{K}_{\mathbb{M}_2(\mathcal{A})}^m\ne 0$ for all $m\geq1$, yielding that  $\mathbf{K}_{\mathbb{M}_2(\mathcal{A})}$ is not Lie nilpotent. The proof of (b) is complete.
		
		(d): Recall that 
		$$
		\mathbf{K}_{\mathbb{M}_2(\mathcal{A})}=\left\lbrace \begin{pmatrix}
			a & b \\
			-b^\natural & c
		\end{pmatrix} \;|\; a,b,c\in \mathcal{A} \text{ and } a=-a^\natural, c=-c^\natural \right\rbrace .
		$$
		We consider the following two possible cases: 
		
		\medskip 
		
		\textbf{Case 3.} \textit{${\rm char}(K)\ne 2$ and the involution $^\natural$ on $\mathcal{A}$ is trivial.} Because $a^\natural=a$, the equation $a=-a^\natural$ is equivalent to $2a=0$, and hence $a=0$. Similarly, we also obtain that $c=0$. Moreover, as $b^\natural=b$, we finally get that 
		$$
		\mathbf{K}_{\mathbb{M}_2(\mathcal{A})}=\left\lbrace \begin{pmatrix}
			0 & b \\
			-b & 0
		\end{pmatrix} \;|\; b\in \mathcal{A} \right\rbrace .
		$$
		 It follows that $[\mathbf{K}_{\mathbb{M}_2(\mathcal{A})},\mathbf{K}_{\mathbb{M}_2(\mathcal{A})}]=0$. 
		 
		 \medskip 
		 
		 \textbf{Case 4.} \textit{${\rm char}(K)\ne 2$ and the involution $^\natural$ on $\mathcal{A}$ is non-trivial.} Fix $a\in \mathcal{A}$ such that $a\ne a^\natural$. Put $u=a-a^\natural$. Then, we have $u=-u^\natural$. We construct inductively a family of matrices $A_m$, $B_m$, and $X_m$ as follows. For $m=1$, we put
		 \[
		 A_1=\left(\begin{array}{cc}
		 	0 & u\\
		 	u & 0
		 \end{array}\right),~
		 B_1=\left(\begin{array}{cc}
		 	u & 0\\
		 	0 & -u
		 \end{array}\right) \text{ and }X_1=[A_1,B_1],
		 \]
		 For $m\geq2$, we put
		 \[
		 A_m=[X_{m-1},B_{m-1}],
		 B_m=[X_{m-1},A_{m-1}] \text{ and }X_m=[A_m,B_m].
		 \]
		 Since $A_1,B_1\in\mathbf{K}_{\mathbb{M}_2(\mathcal{A})}$, we get that $X_1\in\mathbf{K}_{\mathbb{M}_2(\mathcal{A})}^{(1)}$. Inductively, it can be checked that $A_m, B_m\in \mathbf{K}_{\mathbb{M}_2(\mathcal{A})}^{(m-1)}$, from which it follows that $X_m\in  \mathbf{K}_{\mathbb{M}_2(\mathcal{A})}^{(m)}$.
		  We claim that for each positive integer $m$, there exists $v_m\in \mathcal{A}$ such that $0\neq v_m=-v_m^\natural$ and
		 $$ 
		 \begin{aligned}
		 	&A_m=\left(\begin{array}{cc}
		 		0 & v_m\\
		 		v_m & 0
		 	\end{array}\right),\\[0.3cm]
		 	&B_m=\left(\begin{array}{cc}
		 		(-1)^{m+1}v_m & 0\\
		 		0 & (-1)^mv_m
		 	\end{array}\right),\\[0.3cm]
		 	&X_m=\left(\begin{array}{cc}
		 		0 & (-1)^m 2v_m^2\\
		 		(-1)^{m+1}2v_m^2 & 0
		 	\end{array}\right)\neq 0.
		 \end{aligned}
		 $$
		 We proceed by induction on $m$. For $m=1$, we set $v_1=u$. Then, we have $0\ne v_1=v_1^\natural$ and $ X_1=[A_1,B_1]=\left(\begin{array}{cc}
		 0 & -2u^2\\
		 2u^2 & 0
		 \end{array}\right)\neq 0$. Thus, the claim is true for the case $m=1$. Now, assume that the claim is true for $m=i\geq 2$. That means there exist $v_i\in \mathcal{A}$ such that $0\neq v_i=-v_i^\natural $ and
		 $$ 
		 \begin{aligned}
		 	&A_i=\left(\begin{array}{cc}
		 		0 & v_i\\
		 		v_i & 0
		 	\end{array}\right),\\[0.3cm]
		 	&B_i=\left(\begin{array}{cc}
		 		(-1)^{i+1}v_i & 0\\
		 		0 & (-1)^iv_i
		 	\end{array}\right),\\[0.3cm]
		 	&X_i=\left(\begin{array}{cc}
		 		0 & (-1)^i 2v_i^2\\
		 		(-1)^{i+1}2v_i^2 & 0
		 	\end{array}\right)\neq 0.
		 \end{aligned}
		 $$
		 Now, we prove the claim for the case $m=i+1$. A straightforward calculation shows that 
		  $$ 
		 \begin{aligned}
		 	&A_{i+1}=\left(\begin{array}{cc}
		 		0 & 4v_i^3\\
		 		4v_i^3 & 0
		 	\end{array}\right),\\[0.3cm]
		 	&B_{i+1}=\left(\begin{array}{cc}
		 		(-1)^i4v_i^3 & 0\\
		 		0 & (-1)^{i+1}4v_i^3
		 	\end{array}\right),\\[0.3cm]
		 	&X_{i+1}=\left(\begin{array}{cc}
		 		0 & (-1)^{i+1} 32v_i^6\\
		 		(-1)^{i+2} 32v_i^6 & 0
		 	\end{array}\right)\neq 0.
		 \end{aligned}
		 $$
		 Now, put $v_{i+1}=4v_i^3$. Then, we have 
		 $$
		 v_{i+1}^\natural=(4v_i^3)^\natural =4(v_i^\natural)^3=-4v_i^3=-v_{i+1}\ne0.
		 $$
		 Therefore, the claim is true for $m=i+1$. The claim is proved.  It follows that $X_m\neq 0$ for all $m\geq 1$. In other words, we have $\mathbf{K}_{\mathbb{M}_2(\mathcal{A})}^{(m)}$ for all $m\geq 1$, which means that $\mathbf{K}_{\mathbb{M}_2(\mathcal{A})}$ is not Lie solvable.		
	\end{proof}
	
	In the remainder of this section, we record two corollaries of Proposition \ref{proposition _involution} for further use. In the proposition, let $A=K$ and $^\natural$ be the trivial involution on $K$. Then, the extended involution $^\natural$ on $\mathbb{M}_n(K)$ given by (\ref{equation_involution}) is just the transpose operation on $\mathbb{M}_n(K)$; that is, for every matrix $(a_{ij})_{n\times n}\in \mathbb{M}_n(K)$, we have 
	\begin{equation}\label{equation_involution on matrix ring over K}
		\left( (a_{ij})_{n\times n}\right) ^\natural = \left( a_{ji}\right) _{n\times n}.
	\end{equation}
	With respect to this involution, the Lie algebra $\mathbf{K}_{\mathbb{M}_n(K)}$ consists precisely of skew-symmetric matrices in $\mathbb{M}_n(K)$:
	$$
	\mathbf{K}_{\mathbb{M}_n(K)}=\left\lbrace (a_{ij}): a_{ji} = -a_{ij} \text{ for all } 1\leq i,j\leq n\right\rbrace .
	$$
	As a consequence of Proposition \ref{proposition _involution}, we easily obtain the following:
	\begin{corollary}\label{lem.matrix.field}
		Let $K$ be a field, and $n\geq 2$ an integer. Let $\mathbf{K}_{\mathbb{M}_n(K)}$ be the Lie subalgebra of  skew-symmetric matrices of $\mathbb{M}_n(K)^-$. The following assertions hold:
		\begin{enumerate}[font=\normalfont]
			\item\label{lem.matrix.field.item1} Suppose that $n=2$. If ${\rm char}(K)\ne 2$, then $[\mathbf{K}_{\mathbb{M}_n(K)},\mathbf{K}_{\mathbb{M}_n(K)}]=0$; and, if ${\rm char}(K) = 2$, then $\mathbf{K}_{\mathbb{M}_n(K)}$ is Lie solvable of index $2$ but not Lie nilpotent. 
			\item\label{lem.matrix.field.item2} If $n\geq 3$, then  $\mathbf{K}_{\mathbb{M}_n(K)}$ is not Lie solvable.
		\end{enumerate}
	\end{corollary}
	
	\medskip 
	
	For any field $K$, let $K[x,x^{-1}]$ be the \textit{Laurent polynomial $K$-algebra} in $x$; that is,
	$$
	K[x,x^{-1}]=\left\lbrace \sum_{i=k}^{n}a_ix^i\;|\; a_i\in K \text{ and } k\geq n\in\mathbb{Z}\right\rbrace .
	$$
	Then, the map $^\natural: K[x,x^{-1}]\to K[x,x^{-1}]$ defined by $(\sum_{i}a_ix^i)^\natural=\sum_{i}a_ix^{-i}$ is a non-trivial involution on $K[x,x^{-1}]$. Now, in Proposition \ref{proposition _involution}, we consider the case where $\mathcal{A}=K[x,x^{-1}]$ and extend the involution $^\natural$ on $K[x,x^{-1}]$ to the matrix ring $\mathbb{M}_n(K[x,x^{-1}])$ as defined in (\ref{equation_involution}). This extended involution $^\natural: \mathbb{M}_n(K[x,x^{-1}])\to \mathbb{M}_n(K[x,x^{-1}])$ is precisely given by the following rule:
	\begin{equation}\label{equation_involution on Laurent polynomial}
		\left( (f_{ij}(x))_{n\times n}\right) ^\natural = \left( f_{ji}(x^{-1})\right) _{n\times n}
	\end{equation}
	for each matrix $(f_{ij}(x))\in \mathbb{M}_n(K[x,x^{-1}])$. With respect to this involution, we easily get the following corollary to Proposition \ref{proposition _involution}:
	
	\begin{corollary}\label{lem.matrix.LaurentPolyRing}
			Let $K[x,x^{-1}]$ be the \textit{Laurent polynomial $K$-algebra}. Let $\mathbf{K}_{\mathbb{M}_n(K[x,x^{-1}])}$ be the Lie subalgebra of skew-symmetric elements of $\mathbb{M}_n(K[x,x^{-1}])^-$ with respect to the involution defined by $($\ref{equation_involution on Laurent polynomial}$)$.  Then, the following assertions hold:
				\begin{enumerate}[font=\normalfont]
						\item\label{lem.matrix.LaurentPolyRing.i} If $\mathrm{char}(K)=2$, then $\mathbf{K}_{\mathbb{M}_2(K[x,x^{-1}])}$ is Lie solvable of index $3$ but not Lie nilpotent.
						\item\label{lem.matrix.LaurentPolyRing.ii} If $\mathrm{char}(K)\neq 2$, then $\mathbf{K}_{\mathbb{M}_2(K[x,x^{-1}])}$ is not Lie solvable.
			\end{enumerate}
	\end{corollary}
	
	\section{Lie Solvability and Jordan Solvability of $L_K(E)$}
	
	 Throughout this section, we will denote by $\mathbf{K}_{L_K(E)}$ and $\mathbf{S}_{L_K(E)}$ the Lie subalgebra of $L_K(E)^-$ and the Jordan subalgebra of $L_K(E)^\circ$with respect to standard involution $^\star$ on $L_K(E)$  as defined before, respectively. The main aim of this section is to provide a necessary and sufficient condition on the graph $E$ and the field $K$ for $\mathbf{K}_{L_K(E)}$ and $\mathbf{S}_{L_K(E)}$ to be Lie and Jordan solvable, respectively. We start this section with a lemma which determines a $K$-base of $L_K(E)$:
	\begin{lemma}[{\cite[Corollary 1.5.12]{Bo_abrams_2017}}]\label{lemma_basic}
		Let $E$ be a graph and $K$ a field. Put $\mathscr{A}=\{\lambda\nu^*\;|\; \lambda,\nu\in{\rm Path}(E) \text{ and }r(\lambda)=r(\nu)\}$. For each $v\in{\rm Reg}(E)$, let $\{e_1^v,\dots,e_{n_v}^v\}$ be an enumeration of the elements of $s^{-1}(v)$. Then, a $K$-basis of $L_K(E)$ is given by the family
		$$
		\mathscr{B}=\mathscr{A}\backslash\{\lambda e_{n_v}^v(e_{n_v}^v)^*\nu^*\;:\;r(\lambda)=r(\nu)=v\in{\rm Reg}(E)\}.
		$$
	\end{lemma}

	The next lemma provides a necessary condition for $\mathbf{K}_{L_K(E)}$ to be Lie solvable.  The ideas used to prove this lemma come from the proof of \cite[Lemma 2.2 and Theorem 2.3]{nam-zhang-22}. It can be said that, in $L_K(E)$, we can easily choose a suitable subalgebra which is isomorphic to the matrix ring $\mathbb{M}_3(K)$. Then, Corollary \ref{lem.matrix.field} is applied to get all possible cases for $\mathbf{K}_{L_K(E)}$ to be Lie solvable. For convenience of the readers, we still include the proof here.  
	\begin{lemma}\label{lemma_solvability_matrix}
		Let $L_K(E)$ be the Leavitt path algebra of a graph $E$ with coefficients in a field $K$. Then, $\mathbf{K}_{L_K(E)}$ is not Lie solvable if one of the following cases occurs:
		\begin{enumerate}[font=\normalfont]
			\item $E$ contains a cycle with an exit.
			\item $E$ contains one of following subgraphs:
				
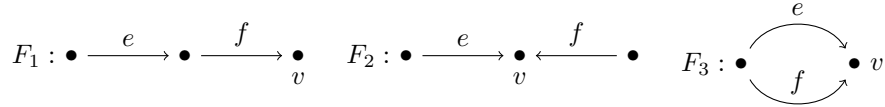
\begin{figure}[!htb]
				\begin{minipage}{0.33\textwidth}
					\begin{tikzpicture}
						\node at (0,0) (1) {$F_1:$};
						\node at (0.5,0) (0) {$\bullet$};
						\node at (2,0) (1) {$\bullet$};
						\node at (3.5,0) (2) {$\bullet$};
						\node at (3.5,-0.3) {$v$};
						
						\draw [->] (0) to node[above] {$e$} (1);
						\draw [->] (1) to node[above] {$f$} (2);
					\end{tikzpicture}
				\end{minipage}\hfill
				\begin{minipage}{0.33\textwidth}
					\begin{tikzpicture}
						\node at (0,0) (1) {$F_2:$};
						\node at (0.5,0) (0) {$\bullet$};
						\node at (2,0) (1) {$\bullet$};
						\node at (3.5,0) (2) {$\bullet$};
						\node at (2,-0.3) {$v$};
						
						\draw [->] (0) to node[above] {$e$} (1);
						\draw [->] (2) to node[above] {$f$} (1);
					\end{tikzpicture}
				\end{minipage}\hfill
				\begin{minipage}{0.3\textwidth}
				\begin{tikzpicture}
					\node at (0,0) (1) {$F_3:$};
					\node at (0.5,0) (0) {$\bullet$};
					\node at (2,0) (2) {$\bullet$};
					\node at (2.3,0) {$v$};
					
					\draw [->] (0) to [out=60,in=120,looseness=1] node[above] {$e$} (2);
					\draw [->] (0) to [out=-60,in=-120,looseness=1] node[above] {$f$} (2);
				\end{tikzpicture}
				\end{minipage}
			\end{figure}
		\end{enumerate}
	\end{lemma}
	
	\begin{proof}
	(1) Assume that $E$ contains a cycle $c$ with an exit $f$. Then, according to Lemma \ref{lemma_basic}, the set $\{c^mff^*(c^*)^n: 1\leq m,n\leq 3\}$ is linearly independent over $K$ satisfying the rule $$(c^mff^*(c^*)^i)(c^jff^*(c^*)^n)=\delta_{ij}c^mff^*(c^*)^n,$$
 	where $\delta_{ij}$ is the Kronecker delta. It follows that the $K$-subalgebra of $L_K(E)$, say $A$, generated by $\{c^mff^*(c^*)^n: 1\leq m,n\leq 3\}$ is isomorphic to $\mathbb{M}_3(K)$. The restriction of the standard involution $^\star$ on $L_K(E)$ to $A$ is still an involution on $A$ which satisfies  $\mathbf{K}_{A}\subseteq \mathbf{K}_{L_K(E)}$. Let $^\natural$ be the involution on $\mathbb{M}_3(K)$ as defined in (\ref{equation_involution on matrix ring over K}). Then, it can be checked that ($A,^\star$) is isomorphic to $(\mathbb{M}_3(K),^\natural)$ as $K$-algebras with involution. Thus, we have that $\mathbf{K}_{A}$ is isomorphic to  $\mathbf{K}_{\mathbb{M}_3(K)}$ as Lie algebras over $K$. Because $\mathbf{K}_{\mathbb{M}_3(K)}$ is not Lie solvable (see Corollary \ref{lem.matrix.field}), it follows that $\mathbf{K}_{A}$, and hence $\mathbf{K}_{L_K(E)}$ is not Lie solvable.
 
 	(2) For the proof of this assertion, we assume first that $E$ contains $F_1$. Then, $v, f, ef$ are all paths ending at $v$. It is straightforward to check the subalgebra $B$ of $L_K(E)$ generated by $v$, $f$, $ef$, $f^*$, $f^*e^*$, $ff^*$, $eff^*$, $ff^*e^*$, $eff^*e^*$ is isomorphic to $\mathbb{M}_3(K)$ via map: $eff^*e^*\mapsto E_{11}$, $ff^*\mapsto E_{22}$, $v\mapsto E_{33}$, $eff^*\mapsto E_{12}$, $ff^*e^*\mapsto E_{21}$, $f\mapsto E_{23}$, $f^*\mapsto E_{32}$, $ef\mapsto E_{13}$, and $f^*e^*\mapsto E_{31}$ (see also a similar techniques used in the proof of  \cite[Lemma 2.6.4]{Bo_abrams_2017}). The restriction of the involution $^\star$ to $B$ is also an involution on $B$ for which $\mathbf{K}_{B}\subseteq \mathbf{K}_{L_K(E)}$. Thus, $\mathbf{K}_{B}$ is isomorphic to $\mathbf{K}_{\mathbb{M}_3(K)}$ as Lie algebras over $K$. Therefore, by Corollary \ref{lem.matrix.field}, we conclude that $\mathbf{K}_{B}$, and so $\mathbf{K}_{L_K(E)}$ is not Lie solvable. 
	
	Finally, assume that $E$ contains $F_2$ or $F_3$. Then, the subalgebra $C$ generated by $v, e, f, e^*, f^*, ee^*, ff^*, ef^*, fe^*$ is isomorphic to $\mathbb{M}_3(K)$ via map $ee^*\mapsto E_{11}$, $ff^*\mapsto E_{22}$, $v\mapsto E_{33}$, $ef^*\mapsto E_{12}$, $fe^*\mapsto E_{21}$, $e\mapsto E_{13}$, $e^*\mapsto E_{31}$, $f\mapsto E_{23}$, $f^*\mapsto E_{32}$. A similar argument as in previous paragraph shows that $\mathbf{K}_{L_K(E)}$ is not Lie solvable too. The proof is now completed.
	\end{proof}
	As a consequence of Lemma \ref{lemma_solvability_matrix}, we have the following remark.
	\begin{remark}\label{corollary_E_1-E_6}
		Let $L_K(E)$ be the Leavitt path algebra of a graph $E$ with coefficients in a field $K$. If $\mathbf{K}_{L_K(E)}$ is Lie solvable then $E$ is a disjoint union of following graphs:
		
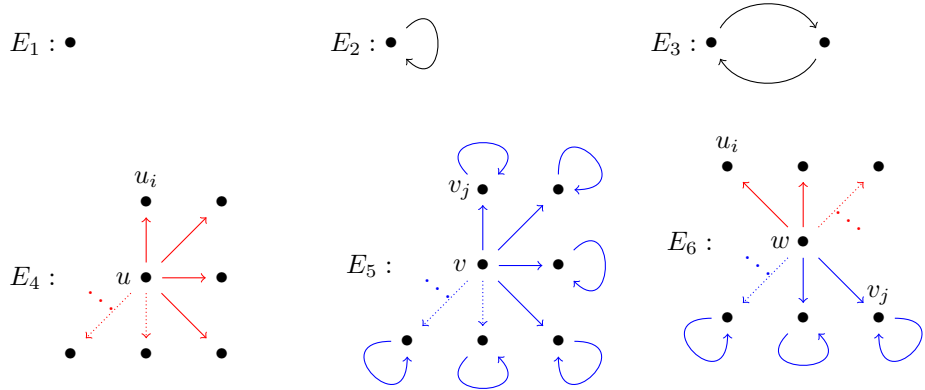
\begin{figure}[!htb]
			\begin{minipage}{0.33\textwidth}
				\begin{tikzpicture}
					\node at (0,0) (1) {$E_1:$};
					\node at (0.5,0) (0) {$\bullet$};
				\end{tikzpicture}
			\end{minipage}\hfill
			\begin{minipage}{0.33\textwidth}
				\begin{tikzpicture}
					\node at (0,0) (1) {$E_2:$};
					\node at (0.5,0) (0) {$\bullet$};
					\draw [->] (0) to [out=45,in=-45,looseness=5] (0);
				\end{tikzpicture}
			\end{minipage}\hfill
			\begin{minipage}{0.33\textwidth}
				\begin{tikzpicture}
					\node at (0,0) (1) {$E_3:$};
					\node at (0.5,0) (0) {$\bullet$};
					\node at (2,0) (2) {$\bullet$};
					
					\draw [->] (0) to [out=60,in=120,looseness=1] (2);
					\draw [->] (2) to [out=-120,in=-60,looseness=1]  (0);
				\end{tikzpicture}
			\end{minipage}\vfill
			\begin{minipage}{0.33\textwidth}
				\begin{tikzpicture}
					\node at (-1.5,0) (E4) {$E_4:$};
					\node at (0,0) (0;0) {$\bullet$};
					\node at (-0.3,0) (u) {$u$};
					\node at (0,1) (0;1) {$\bullet$};
					\node at (1,1) (1;1) {$\bullet$};
					\node at (1,0) (1;0) {$\bullet$};
					\node at (0,1.3) (u) {$u_i$};
					\node at (-1,-1) (-1;-1) {$\bullet$};
					\node at (-0.6,-0.2) (-0.5;-0.5) {\color{red}$\ddots$};
					\node at (0,-1) (0;-1) {$\bullet$};
					\node at (1,-1) (1;-1) {$\bullet$};
					
					\draw [->, red] (0;0) to (1;0);
					\draw [->, red] (0;0) to (1;1);
					\draw [->, red] (0;0) to (0;1);
					\draw [->, red] (0;0) to (1;-1);
					\draw [->, red,densely dotted] (0;0) to (0;-1);
					\draw [->, red,densely dotted] (0;0) to (-1;-1);
				\end{tikzpicture}
			\end{minipage}\hfill
			\begin{minipage}{0.33\textwidth}
				\begin{tikzpicture}
					\node at (-1.5,0) (E4) {$E_5:$};
					\node at (0,0) (0;0) {$\bullet$};
					\node at (-0.3,0) (v) {$v$};
					\node at (0,1) (0;1) {$\bullet$};
					\node at (1,1) (1;1) {$\bullet$};
					\node at (1,0) (1;0) {$\bullet$};
					\node at (-0.3,1) (v) {$v_j$};
					\node at (-1,-1) (-1;-1) {$\bullet$};
					\node at (-0.6,-0.2) (-0.5;-0.5) {\color{blue}$\ddots$};
					\node at (0,-1) (0;-1) {$\bullet$};
					\node at (1,-1) (1;-1) {$\bullet$};
					
					\draw [->, blue] (0;0) to (1;0);
					\draw [->, blue] (1;0) to [out=45,in=-45,looseness=5] (1;0);
					\draw [->, blue] (0;0) to (1;1);
					\draw [->, blue] (1;1) to [out=90,in=0,looseness=8] (1;1);
					\draw [->, blue] (0;0) to (0;1);
					\draw [->, blue] (0;1) to [out=135,in=45,looseness=5] (0;1);
					\draw [->, blue] (0;0) to (1;-1);
					\draw [->, blue] (1;-1) to [out=0,in=-90,looseness=8] (1;-1);
					\draw [->, blue,densely dotted] (0;0) to (0;-1);
					\draw [->, blue] (0;-1) to [out=-135,in=-45,looseness=5] (0;-1);
					\draw [->, blue,densely dotted] (0;0) to (-1;-1);
					\draw [->, blue] (-1;-1) to [out=-180,in=-90,looseness=8] (-1;-1);
				\end{tikzpicture}
			\end{minipage}\hfill
			\begin{minipage}{0.33\textwidth}
				\begin{tikzpicture}
					\node at (-1.5,0) (E6) {$E_6:$};
					\node at (0,0) (0;0) {$\bullet$};
					\node at (-0.3,0) (w) {$w$};
					\node at (0,1) (0;1) {$\bullet$};
					\node at (1,1) (1;1) {$\bullet$};
					\node at (0.6,0.4) (rddots) {\color{red}$\ddots$};
					\node at (-1,1.3) (u) {$u_i$};
					\node at (-1,1) (-1;1) {$\bullet$};
					\node at (-1,-1) (-1;-1) {$\bullet$};
					\node at (1,-0.7) (v) {$v_j$};
					\node at (-0.6,-0.2) (lddots) {\color{blue}$\ddots$};
					\node at (0,-1) (0;-1) {$\bullet$};
					\node at (1,-1) (1;-1) {$\bullet$};
					
					\draw [->,red,densely dotted] (0;0) to (1;1);
					\draw [->,red] (0;0) to (-1;1);
					\draw [->,red] (0;0) to (0;1);
					\draw [->,blue] (0;0) to (1;-1);
					\draw [->,blue] (1;-1) to [out=0,in=-90,looseness=8] (1;-1);
					\draw [->,blue] (0;0) to (0;-1);
					\draw [->,blue] (0;-1) to [out=-135,in=-45,looseness=5] (0;-1);
					\draw [->,blue,densely dotted] (0;0) to (-1;-1);
					\draw [->,blue] (-1;-1) to [out=-180,in=-90,looseness=8] (-1;-1);
				\end{tikzpicture}
			\end{minipage}
		\end{figure}\\
		
	\noindent in which $\{u_i|i\in I\}$ and $\{v_j|j\in J\}$, where $I$ and $J$ are indexing non-empty sets of arbitrary cardinality.
	\end{remark}
	
	\begin{remark}\label{remark_finite cases}
		Let $E_4, E_5, E_6$ be graphs as defined above. If $m:=|I|$ and $n:=|J|$ are finite, then by \cite[Corollary 2.7.5]{Bo_abrams_2017}, we get that 
		$$ 
		\begin{aligned}
			& L_K(E_1)\cong K,\;\; L_K(E_2)\cong K[x,x^{-1}], \;\; L_K(E_3)\cong \mathbb{M}_2(K[x,x^{-1}]),\\
			& L_K(E_4)\cong \bigoplus_{i=1}^n \mathbb{M}_2(K),\;\; L_K(E_5)\cong \bigoplus_{i=1}^m \mathbb{M}_2(K[x,x^{-1}]),\\
			&  L_K(E_6)\cong\bigoplus_{i=1}^n \mathbb{M}_2(K) \oplus \bigoplus_{i=1}^m \mathbb{M}_2(K[x,x^{-1}]).
		\end{aligned}
		$$
	\end{remark}
	
	In what follows, we give a sufficient condition for $\mathbf{K}_{L_K(E)}$ to be Lie solvable. We proceed by proving several auxiliary lemmas.
	\begin{lemma}\label{lemma_sum}
		Let $\mathcal{A}$ be a $K$-algebra with an involution $\star$, and $\mathbf{J}$ be an indexing set of arbitrary cardinality. Let $I_j, j\in \mathbf{J}$, be ideals of $\mathcal{A}$ satisfying the property that for each $j$, if $a\in I_j$ then $a^{\star}\in I_j$. Put $M=\sum_{j\in\mathbf{J}}I_j$. If $M=\bigoplus_{j\in\mathbf{J}}I_j$, then $\mathbf{K}_M=\bigoplus_{j\in\mathbf{J}}\mathbf{K}_{I_j}$. Consequently, $\mathbf{K}_M$ is Lie solvable $($nilpotent$)$ if and only if $\mathbf{K}_{I_j}$ is Lie solvable $($nilpotent$)$ of a common index for all $j\in\mathbf{J}$.
	\end{lemma}
	\begin{proof}
		Let $x\in \mathbf{K}_M$. Then, we may write $x=a_{j_1}+a_{j_2}+\dots+a_{j_n}$, where $a_{j_k}\in I_{j_k}$ for all $1\leq k\leq n$. The equation $x^{\star}=-x$ implies that 
		$$
		a_{j_1}+a_{j_1}^{\star}=-(a_{j_2}+a_{j_2}^{\star}+\cdots+a_{j_n}+a_{j_n}^{\star}).
		$$
		Because $a_{j_k}\in I_{j_k}$ for all $i$, the above equation shows that $a_{j_1}+a_{j_1}^{\star}\in I_{j_1}\cap (I_{j_2}\oplus\cdots\oplus I_{j_n})=0$. It follows that $a_1+a_1^{\star}=0$, and so $a_{j_1}\in \mathbf{K}_{I_{j_1}}$. Similarly, we also obtain that $a_i\in \mathbf{K}_{I_{j_k}}$ for $k\in\{2,\dots,n\}$. In other words, we have $\mathbf{K}_M\subseteq\sum_{j\in\mathbf{J}}\mathbf{K}_{I_j}=\bigoplus_{j\in\mathbf{J}}\mathbf{K}_{I_j}$. The inverse inclusion is clear. The final assertion follows from the fact that $\mathbf{K}_M^{(n)}=\bigoplus_{j\in\mathbf{J}}\mathbf{K}_{I_j}^{(n)}$ and $\mathbf{K}_M^n=\bigoplus_{j\in\mathbf{J}}\mathbf{K}_{I_j}^n$ for all integer $n\geq1$.
	\end{proof}
	
	\begin{lemma}\label{lemma_Liesolvable of E1,2,3}
		Let $K$ be a field, and $E_1, E_2, E_3$ be graphs as defined in Remark \ref{corollary_E_1-E_6}. Then, the following assertions hold:
		\begin{enumerate}[font=\normalfont]
			\item $\mathbf{K}_{L_K(E_1)}$ and $\mathbf{K}_{L_K(E_2)}$ are Lie solvable of index $1$.
			\item $\mathbf{K}_{L_K(E_3)}$ is Lie solvable $($of index $3$$)$ if and only if $\mathrm{char}(K)=2$. Additionally, $\mathbf{K}_{L_K(E_3)}$ is not Lie nilpotent.
		\end{enumerate}
	\end{lemma}
	\begin{proof}
		The assertion (i) follows immediately from the fact that $L_K(E_1)$ and $L_K(E_2)$ are commutative. Moreover, because $L_K(E_3)\cong \mathbb{M}_2(K[x,x^{-1}])$, we get that $\mathbf{K}_{L_K(E_3)}$ is not Lie nilpotent.
	\end{proof}
	
	\begin{lemma}\label{lemma_Liesolvable of E4,5,6}
		Let $K$ be a field, and $E_4, E_5, E_6$ be graphs as defined in Remark \ref{corollary_E_1-E_6}. Then, the following assertions hold:
		\begin{enumerate}[font=\normalfont]
			\item If $\mathbf{K}_{L_K(E_4)}$ is Lie solvable, then it is Lie solvable  of index $l\leq3$. More precisely, if $|I|$ is finite and ${\rm char}(K)\ne 2$ $($resp.,  ${\rm char}(K) = 2$$)$, then $l=1$ $($resp., $l=2$$)$. If $|I|$ is infinite and ${\rm char}(K)\ne 2$ $($reps., ${\rm char}(K) = 2$$)$, then $l=2$ $($resp., $l=3$$)$. Additionally, $\mathbf{K}_{L_K(E_4)}$ is Lie nilpotent if and only if $|I|$ is finite and ${\rm char}(K)\ne 2$, if and only if $[\mathbf{K}_{L_K(E_4)},\mathbf{K}_{L_K(E_4)}]=1$.
			\item $\mathbf{K}_{L_K(E_5)}$ is Lie solvable $($of index $l\leq 4$$)$ if and only if ${\rm char}(K)=2$. More precisely, if $|J|$ is finite $($resp.,  infinite$)$, then $l=3$ $($resp., $l=4$$)$. Additionally, $\mathbf{K}_{L_K(E_5)}$ is not Lie nilpotent.
			\item $\mathbf{K}_{L_K(E_6)}$ is Lie solvable $($of index $l\leq 4$$)$ if and only if ${\rm char}(K)=2$. More precisely, if both $|I|$ and $|J|$ are finite $($resp., either $|I|$ or $|J|$ is infinite$)$, then $l=3$ $($resp., $l=4$$)$. Additionally, $\mathbf{K}_{L_K(E_6)}$ is not Lie nilpotent.
		\end{enumerate}
	\end{lemma}
	\begin{proof}
		(1) Consider the graph $E_4$. For each $1\leq i\leq n$, let $I_i$ be the ideal of $L_K(E_4)$ generated by $u_i$. Then, it is clear that, for each $i$,
		$$
		I_i={\rm span}_K\{u_i,e_i,e_i^*, e_ie_i^*\}.
		$$
		Moreover, according to Lemma \ref{lemma_basic}, the set $\{u_i,e_i,e_i^*, e_ie_i^*\}$ is linearly independent over $K$. Thus, it is can be checked that $I_i\cong\mathbb{M}_2(K)$ via the mapping $u_i\mapsto E_{11}$, $e_ie_i^*\mapsto E_{22}$, $e_i \mapsto E_{12}$ and $e_i^*\mapsto E_{21}$. Moreover, we have $\sum_{i=1}^n I_i=\bigoplus_{i=1}^n I_i$. Put $M=\bigoplus_{i=1}^n I_i$. Then, according to Lemma \ref{lemma_sum} and Corollary \ref{lem.matrix.field}, we conclude that $\mathbf{K}_M=\bigoplus_{i=1}^n \mathbf{K}_{I_i}$, and that $\mathbf{K}_M$ is Lie solvable. Moreover, in view of Corollary \ref{lem.matrix.field}, if ${\rm char}(K)\ne 2$, then $[\mathbf{K}_M,\mathbf{K}_M]=0$; and, if ${\rm char}(K) = 2$ then $\mathbf{K}_M$ is Lie solvable of index $2$ but not Lie nilpotent. Now, we consider two possible cases:
		
		\smallskip
		
		\textbf{Case 1.} \textit{$n$ is finite}. In this case, as  $u\in M$, we get that $L_K(E_4) = M$, from which it follows that $\mathbf{K}_{L_K(E_4)}=\mathbf{K}_M$. Thus, all conclusions follow immediately in this case.
		
		\smallskip 
		
		\textbf{Case 2.} \textit{$n$ is infinite}. In this case, $u$ is an infinite emitter and $u\not\in M$. This implies that $L_K(E_4)=Ku \oplus M$ as $K$-vector spaces. It follows that 
		$$[\mathbf{K}_{L_K(E_4)},\mathbf{K}_{L_K(E_4)}]\subseteq[L_K(E_4),L_K(E_4)]=[M,M]\subseteq M.$$ 
		As $M$ is Lie solvable, we conclude that $[\mathbf{K}_{L_K(E_4)},\mathbf{K}_{L_K(E_4)}]$, and hence $\mathbf{K}_{L_K(E_4)}$, is Lie solvable. To determine the solvable index of $\mathbf{K}_{L_K(E_4)}$, we consider the following subcases:
		
		\smallskip 
		
		\textit{Subcase 2.1. ${\rm char}(K)\ne2$}. Because $[\mathbf{K}_{L_K(E_4)},\mathbf{K}_{L_K(E_4)}]\subseteq M$, we conclude that $[\mathbf{K}_{L_K(E_4)},\mathbf{K}_{L_K(E_4)}]\subseteq \mathbf{K}_M$ which is Lie solvable of index 1. This means that $\mathbf{K}_{L_K(E_4)}$ is Lie solvable of index at most $2$. Moreover, as $u,e_1-e_1^*\in \mathbf{K}_{L_K(E_4)}$ satisfying $u(e_1-e_1^*)=e_1\ne-e_1^*=(e_1-e_1^*)u$. That means $[\mathbf{K}_{L_K(E_4)}, \mathbf{K}_{L_K(E_4)}]\ne0$; and so, it is Lie solvable of index $2$. On the other hand, for any $m\geq1$, the element $[e_1-e_1^*,\ce{_m}v]\in \mathbf{K}_{L_K(E_4)}^m$ and $[e_1-e_1^*,\ce{_m}u]=(-1)^me_1-e_1^*\ne0$.  It follows that $\mathbf{K}_{L_K(E_4)}^m\ne0$ for all $m\geq1$. Thus, $\mathbf{K}_{L_K(E_4)}$ is not Lie nilpotent in this case. 
		
		\smallskip 
		
		\textit{Subcase 2.2. ${\rm char}(K)=2$}. As $[\mathbf{K}_{L_K(E_4)},\mathbf{K}_{L_K(E_4)}]\subseteq \mathbf{K}_M$ which is Lie solvable of index $2$, it follows that $\mathbf{K}_{L_K(E_4)}$ is Lie solvable of index at most $3$. Because $v\in \mathbf{K}_{L_K(E_4)}\backslash  \mathbf{K}_M$, we conclude that $\mathbf{K}_{L_K(E_4)}\ne [\mathbf{K}_{L_K(E_4)},\mathbf{K}_{L_K(E_4)}]$. It follows that $\mathbf{K}_{L_K(E_4)}$ is Lie solvable of index $3$. On the other hand,  as $\mathbf{K}_{I_{i}}$ is not Lie nilpotent and $\mathbf{K}_{I_{i}}\subseteq \mathbf{K}_{L_K(E_4)}$ for all $i$, we obtain that $\mathbf{K}_{L_K(E_4)}$ is not Lie nilpotent. The proof of assertion (1) is complete. 
		
		(2) Consider the graph $E_5$. For each $1\leq i\leq m$, let $J_i$ be the ideal of $L_K(E_5)$ generated by the $v_j$'s. Then, we have $J_j\cong\mathbb{M}_2(K[x,x^{-1}])$ for all $j$, and $\sum_{i=1}^n J_j=\bigoplus_{j=1}^n J_j$. Therefore, one can repeat the arguments of the proof of assertion(1), together with using Corollary \ref{lem.matrix.LaurentPolyRing} instead of Corollary \ref{lem.matrix.field}, to get (2).
		
		(3) Consider the graph $E_6$. Let $I_i$ and $J_j$ be respectively the ideals of  $L_K(E_6)$ generated by $v_i$ and $u_j$ for $1\leq i\leq n$ and $1\leq j\leq m$. Then, for each $i$ and each $j$, we have $I_i\cong \mathbb{M}_2(K)$ and $J_j\cong \mathbb{M}_2(K[x,x^{-1}])$. It can be checked that 
		$$
		\sum_{i=1}^n I_i+\sum_{j=1}^m J_j=\bigoplus_{i=1}^n I_i\oplus \bigoplus_{j=1}^m J_j.
		$$
		Put $M=\bigoplus_{i=1}^n I_i$ and $N=\bigoplus_{j=1}^m J_j$. Then, as before, $\mathbf{K}_M$ is Lie solvable of index at most $3$; and, $\mathbf{K}_N$ is Lie solvable (of index at most $4$) if and only if ${\rm char}(K)=2$. Moreover, it is clear that $L_K(E_6)=Kw+(M\oplus N)$. As $N$ is not Lie nilpotent, we get that  $\mathbf{K}_{L_K(E_6)}$ is not Lie nilpotent. Moreover, we have that $\mathbf{K}_{L_K(E_6)}$ is Lie solvable if and only if both $\mathbf{K}_M$ and $\mathbf{K}_N$ are Lie solvable, if and only if ${\rm char}(K)=2$. Now, assume that ${\rm char}(K)=2$. To determine the Lie solvability index of $\mathbf{K}_{L_K(E_6)}$, we consider the following cases. If both $|I|$ and $|J|$ are finite, then $w\in M\oplus N$; and so, $L_K(E_6)=M\oplus N$. Again, by Lemma \ref{lemma_sum}, we obtain that $\mathbf{K}_{L_K(E_6)}=\mathbf{K}_M\oplus \mathbf{K}_N$. According to Corollary \ref{lem.matrix.field}, $\mathbf{K}_M$ is Lie solvable of index $2$, while by Corollary \ref{lem.matrix.LaurentPolyRing}, $\mathbf{K}_N$ is Lie solvable of index $3$. It follows that $\mathbf{K}_{L_K(E_6)}$ is Lie solvable of index $3$. If either $|I|$ or $|J|$ is infinite, then $w\notin M\oplus N$. Thus, a similar argument as used in \textit{Case 2} above shows that $\mathbf{K}_{L_K(E_6)}$ is Lie solvable of index $4$. 
	\end{proof}
	
	We now present the main result of this section: a sufficient and necessary condition for $\mathbf{K}_{L_K(E)}$ to be Lie solvable.
	
	\begin{theorem}\label{theorem_solvable}
		Let $K$ be a field and $E$ a graph. Then, the followings hold:
		\begin{enumerate}[font=\normalfont]
			\item If ${\rm char}(K)=2$, then $\mathbf{K}_{L_K(E)}$ is Lie solvable $($of index $\leq4$$)$ if and only if $E$ is a disjoint union of graphs $E_1,\dots, E_6$. Specially, the index is $4$ precisely when $E$ contains either $E_5$ or $E_6$ with infinite $|I|$ or $|J|$.
			\item If ${\rm char}(K)\ne2$, then $\mathbf{K}_{L_K(E)}$ is Lie solvable $($of index $\leq2$$)$ if and only if $E$ is a disjoint union of graphs $E_1$, $E_2$ and $E_4$.  Specially, the index is $2$ precisely when $E$ contains $E_4$ with infinite $|I|$.
		\end{enumerate}
	\end{theorem}
	\begin{proof}
		
		To prove (1), we assume first that ${\rm char}(K)=2$. For a proof of the ``only if'' part, we assume that $L_K(E)$ is Lie solvable. Then by Remark \ref{corollary_E_1-E_6}, the graph $E$ is a disjoint union of graphs $E_1,\dots, E_6$. Conversely, we assume that $E$ is a disjoint union of graphs $E_1,\dots, E_6$. Then, $L_K(E)$ is a direct summand of ideals each of which is isomorphic to $L_K(E_1)$, $\dots$, $L_K(E_6)$. In view of Lemmas \ref{lemma_Liesolvable of E1,2,3} and \ref{lemma_Liesolvable of E4,5,6}, we get that $L_K(E_1)$, $\dots$, $L_K(E_6)$ are Lie solvable, and so the assertion (1) follows. Finally, the proof of (2) is similar to that of (1).
	\end{proof}
	If $E$ is row-finite, then by Remark \ref{remark_finite cases},  we can obtain more detailed descriptions as follows:
	\begin{corollary}\label{corollary_LieSol_rowfinite}
		Let $K$ be a field and $E$ a row-finite graph.
		\begin{enumerate}[font=\normalfont]
			\item If ${\rm char}(K)=2$, then $\mathbf{K}_{L_K(E)}$ is Lie solvable $($of index $\leq3$$)$ if and only if $E$ is a disjoint union graphs $E_1,\dots, E_6$ with finite $|I|$ and $|J|$. 
			In this case, 
			$$
			L_K(E)\cong \bigoplus_{I_1}K\oplus \bigoplus_{I_2}K[x,x^{-1}] \oplus \bigoplus_{I_3}\mathbb{M}_2(K)\oplus \bigoplus_{I_4}\mathbb{M}_2(K[x,x^{-1}]),
			$$
			for suitable $($possibly infinite$)$ indexing sets $I_1$, $I_2$, $I_3$, $I_4$. Moreover, the index is $3$ precisely when $E$ contains $E_3$, $E_5$ or $E_6$.
			 
			\item If ${\rm char}(K)\ne2$, then $\mathbf{K}_{L_K(E)}$ is Lie solvable $($of index $\leq2$$)$ if and only if $E$ is a disjoint union of graphs $E_1$, $E_2$ and $E_4$ with finite $|I|$. In this case, we have 
			$$
			L_K(E)\cong \bigoplus_{J_1}K\oplus \bigoplus_{J_2}K[x,x^{-1}] \oplus \bigoplus_{J_3}\mathbb{M}_2(K),
			$$
			for suitable $($possibly infinite$)$ indexing sets $J_1$, $J_2$, $J_3$. Moreover, the index is $2$ precisely when $E$ contains $E_3$.
		\end{enumerate}
	\end{corollary}
	The following theorem provides a necessary and sufficient condition for $\mathbf{K}_{L_K(E)}$ to be Lie nilpotent. The proof of this theorem is similar to that of Theorem \ref{theorem_solvable}, so it should be omitted.
	\begin{theorem}\label{theorem_nilpotent}
		Let $K$ be a field and $E$ a graph. 
		\begin{enumerate}[font=\normalfont]
			\item If ${\rm char}(K)=2$, then $\mathbf{K}_{L_K(E)}$ is Lie nilpotent if and only if $E$ is a disjoint union graphs $E_1$ and $E_2$.
			\item If ${\rm char}(K)\ne2$, then $\mathbf{K}_{L_K(E)}$ is Lie nilpotent if and only if $E$ is a disjoint union graphs $E_1$, $E_2$, and $E_4$ with finite $n$.  
		\end{enumerate}
		In particular, $\mathbf{K}_{L_K(E)}$ is Lie nilpotent if and only if $[\mathbf{K}_{L_K(E)}, \mathbf{K}_{L_K(E)}]=0$.
	\end{theorem}

	\begin{corollary}
		Let $K$ be a field and $E$ a graph. If $E$ contains an infinite emitter, then $\mathbf{K}_{L_K(E)}$ is not Lie nilpotent.
	\end{corollary}
	
	The following corollary addresses a situation where being Lie solvable and being Lie nilpotent are equivalent.
	
	\begin{corollary}\label{corollary_charkhac2}
		Let $K$ be a field and $E$ a row-finite graph.  If ${\rm char}(K)\ne2$, then $\mathbf{K}_{L_K(E)}$ is Lie solvable if and only if it is Lie nilpotent, if and only if $E$ is a disjoint union of graphs $E_1$, $E_2$ and $E_4$.
	\end{corollary}
	\begin{proof}
		Assume that $L_K(E)$ is Lie solvable. In view of Corollary \ref{corollary_LieSol_rowfinite}, $E$ is a disjoint union of a finite number of graphs $E_1$, $E_2$ or $E_4$ with finite $|I|$. Therefore, $L_K(E)$ is Lie nilpotent by Theorem \ref{theorem_nilpotent}. 
	\end{proof}
	
	In the remainder of this section, we investigate the solvability of the Jordan algebra $\mathbf{S}_{L_K(E)}$. 
	\begin{theorem}\label{theorem_Jordan_sol}
		Let $K$ be a field and $E$ a graph. Then, the following assertions hold:
		\begin{enumerate}[font=\normalfont]
			\item If ${\rm char}(K)=2$, then $\mathbf{S}_{L_K(E)}$ is Jordan solvable $($resp., nilpotent$)$ if and only if $E$ is a disjoint union of graphs $E_1,\dots, E_6$ $($resp., $E_1$ and $E_2$$)$. 
			\item If ${\rm char}(K)\ne2$, then $\mathbf{S}_{L_K(E)}$ is not Jordan nilpotent. 
		\end{enumerate}
	\end{theorem}
	\begin{proof}
		(1) Assume that $\mathrm{char}(K)=2$. Then,
		\[
		a\circ b=ab+ba=ab-ba=[a,b] \text{ for any }a, b\in L_K(E).
		\]
		It follows that $\mathbf{S}_{L_K(E)}=\mathbf{K}_{L_K(E)}$. Thus, $\mathbf{S}_{L_K(E)}$ is Jordan solvable (resp., nilpotent) if and only if $\mathbf{K}_{L_K(E)}$ is Lie solvable (resp., nilpotent). By Theorems \ref{theorem_solvable} and \ref{theorem_nilpotent}, all conclusions of (1) are asserted.
		
		(2) Assume that $\mathrm{char}(K)\neq 2$. Let $v$ be an arbitrary vertex of $E$. Then, it is clear that $v\in \mathbf{S}_{L_K(E)} $. Put $X_0=v$ and $X_{m+1}=X_m\circ v$, for all $m\geq 0$. Then, it can be checked that $X_m=2^mv$ for all $m\geq 0$. Since $\mathrm{char}(K)\neq 2$, it follows that $X_{m}\neq 0$  for all $m\geq 0$. Hence, $\mathbf{S}_{L_K(E)}$ is not Jordan nilpotent.
	\end{proof}
	To conclude this paper, we relate the solvability of the four structures $L_K(E)^\circ$, $\mathbf{S}_{L_K(E)}$, $L_K(E)^-$ and $\mathbf{K}_{L_K(E)}$.
	\begin{remark}
		Theorem \ref{theorem_Jordan_sol} is also true if we replace $\mathbf{S}_{L_K(E)}$ by the whole $L_K(E)^\circ$. This means $L_K(E)^\circ$ is Jordan solvable or Jordan nilpotent if and only if $\mathbf{S}_{L_K(E)}$ is. At the other extreme, in view of Theorem \ref{theorem_solvable}(2), if ${\rm char}(K)\ne2$ then $\mathbf{K}_{L_K(E_4)}$ is Lie solvable but $L_K(E_4)^-$ is not Lie solvable. These facts show that $\mathbf{K}_{L_K(E)}$ is Lie solvable for a wider class of graph $E$ than $L_K(E)^-$. Also, in some senses, the Lie algebra $\mathbf{S}_{L_K(E)}$ is bigger than the Jordan algebra $\mathbf{K}_{L_K(E)}$ in $L_K(E)$.
	\end{remark}
	
	{\noindent\textbf{Funding.} } This work is funded by Vietnam National Foundation for Science and Technology Development (NAFOSTED) under Grant No. 101.04-2025.41.

\end{document}